\documentclass{amsproc}
\usepackage{eurosym}
\usepackage{amssymb}
\usepackage{amsfonts}
\usepackage{cmtt}

\usepackage{calrsfs}
\DeclareMathAlphabet{\pazocal}{OMS}{zplm}{m}{n}
\newcommand{\Ha}{\mathcal{H}}

\theoremstyle{plain}
\newtheorem{theorem}{Theorem}

\newtheorem{conjecture}[theorem]{Conjecture}
\newtheorem{corollary}[theorem]{Corollary}

\theoremstyle{remark}
\newcommand{\field}[1]{\mathbb{#1}}

\newcommand{\N}{\field{N}}

\newcommand{\R}{\field{R}}

\DeclareMathOperator{\Cut}{Cut}

\newcommand{\superscript}[1]{\ensuremath{^{\textrm{#1}}}}

\def\wu{\superscript{*}}
\def\wg{\superscript{\dag}}

\begin{document}

\title{A note on seminormality of cut polytopes}

\author[Micha{\l} Laso{\'n}]{Micha{\l} Laso{\'n}\wu\footnote{\wu michalason@gmail.com; Institute of Mathematics of the Polish Academy of Sciences, ul.\'{S}niadeckich 8, 00-656 Warszawa, Poland}}

\author[Mateusz Micha{\l}ek]{Mateusz Micha{\l}ek\wg\footnote{\wg wajcha2@poczta.onet.pl; University of Konstanz, Germany}}

\thanks{Micha{\l} Laso\'{n} was supported by Polish National Science Centre grant no. 2019/34/E/ST1/00087.}


\begin{abstract}
	We prove that seminormality of cut polytopes is equivalent to normality.
	This settles two conjectures regarding seminormality of cut polytopes.
\end{abstract}

\maketitle



A \emph{cut} $A\vert B$ in a graph $G(V,E)$ is an unordered partition ($A\vert B=B\vert A$) of its vertices, i.e. $A\sqcup B=V$. A cut $A\vert B$ determines a point in $\R^{\vert E\vert}$, denoted by $\delta_{A\vert B}$, with value $1$ on edges $e$ separated by the cut (i.e. $\vert e\cap A\vert=1$ and $\vert e\cap B\vert=1$) and value $0$ on edges $e$ within cut parts (i.e.~$e\subset A$ or $e\subset B$). The \emph{cut polytope} $\Cut^\square(G)$ corresponding to a graph $G$ is the convex hull of points $\delta_{A\vert B}$ over all cuts $A\vert B$ in $G$, see e.g. \cite{StSu08}.

A polytope $P$ is \emph{normal} if for any $k\in\N$ every lattice point (a point that belongs to the lattice spanned by the lattice points of $P$) in $kP$ is a sum of $k$ lattice points from $P$. A slightly weaker property of a polytope is `very ampleness', see e.g. \cite{LaMi17}. A polytope $P$ is \emph{very ample} if it has only finitely many \emph{gaps} -- lattice points in $kP$ (for some $k$) which are not a sum of $k$ lattice points from $P$. This is equivalent to the fact that for every vertex $v\in P$ the monoid of lattice points in the real cone generated by $P-v$ is generated by lattice points of $P-v$ \cite[Def.~2.2.7]{CoLiSc11}, \cite[Ex.~4.9]{Mi18}. A polytope $P$ is \emph{seminormal} \cite{KoRo20} if for every lattice point $x$, if $2x$ and $3x$ are not gaps, then nor is $x$.

The most well-known conjecture in this area is the following.

\begin{conjecture}[\cite{StSu08}]\label{Conjecture1}
	The cut polytope $\Cut^\square(G)$ is normal if and only if the graph $G$ has no $K_5$ minor.
\end{conjecture}

The implication from the left to the right is known, as $\Cut^\square(K_5)$ is not normal \cite{StSu08} and normality of cut polytopes is a minor closed property \cite{Oh10}. The difficulty of the conjecture lies in proving that for graphs with no $K_5$ minor the cut polytope is normal. We start with a proof of a weaker property -- very ampleness, for which we did not find an explicit reference and which is a corollary of \cite[Corollary $1.3$]{FuGo99}.

\begin{theorem}\label{Theorem1}
	Suppose a graph $G$ has no $K_5$ minor. Then the cut polytope $\Cut^\square(G)$ is very ample.
\end{theorem}

\begin{proof}	
First, we show that the cut polytope is transitive. That is, for any two vertices $v_1,v_2$ of $\Cut^\square(G)$ there exists an affine isomorphism $\varphi$ of $\R^{\vert E\vert}$ such that $\varphi(\Cut^\square(G))=\Cut^\square(G)$ and $\varphi(v_1)=v_2$. It is enough to show that when $v_1=\delta_{\emptyset\vert E}$ and $v_2=\delta_{A\vert B}$ is arbitrary. Map $\varphi_{A\vert B}$ is defined by:
$x_e\rightarrow x_e$ when $e$ is contained in $A$ or $B$, and $x_e\rightarrow 1-x_e$ when $e$ is separated by the cut $A\vert B$. Observe that 
$$\varphi_{A\vert B}(\delta_{C\vert D})=\delta_{(A\cap D)\cup (B\cap C)\vert (A\cap C)\cup (B\cap D)}.$$
In particular, $\varphi_{A\vert B}(\delta_{\emptyset\vert E})=\delta_{A\vert B}$.

Next, we note that for every vertex $\delta_{A\vert B}\in \Cut^\square(G)$ the monoid of lattice points in the real cone generated by $\Cut^\square(G)-\delta_{A\vert B}$ is isomorphic, via $(\varphi_{A\vert B}-\delta_{A\vert B})^{-1}$, to the monoid of lattice points in the real cone generated by $\Cut^\square(G)-\delta_{\emptyset\vert E}=\Cut^\square(G)$. Thus, in order to show that $\Cut^\square(G)$ is very ample it is enough to check the second characterization of very ample polytopes for a single vertex $\delta_{\emptyset\vert E}$. Therefore, very ample property of the cut polytope coincides with the class $\Ha$ in \cite{GoHuDe16} of graphs whose set of cuts is a Hilbert basis in $\R^{\vert E\vert}$.

The statement that remains is proved in \cite[Corollary $1.3$]{FuGo99} and for planar graphs already in \cite{Se79}. Since both rely on the four color theorem, the theorem also does.
\end{proof}

Remark that the cut polytope of $K_5$ is very ample \cite{De82}. Moreover, very ampleness of cut polytopes is not a minor closed property \cite{GoHuDe16}. In particular, Theorem \ref{Theorem1} does not give a characterization of graphs with very ample cut polytopes.

Using Theorem \ref{Theorem1} we settle Conjectures $1.2$ and $4.5$ from \cite{KoRo20}.

\begin{theorem}\label{Theorem2}
	The cut polytope $\Cut^\square(G)$ of a graph $G$ is seminormal if and only if it is normal. In particular, the class of graphs $G$ for which $\Cut^\square(G)$ is seminormal is minor closed.
\end{theorem}

\begin{proof}
	If $\Cut^\square(G)$ is normal, then clearly it is seminormal.
	
	Let $G$ be a graph such that $\Cut^\square(G)$ is seminormal. Then by \cite[Corollary $4.4$]{KoRo20} graph $G$ has no $K_5$ minor. By Theorem \ref{Theorem1} the cut polytope $\Cut^\square(G)$ is very ample. Suppose contrary, that $\Cut^\square(G)$ is not normal -- it has gaps. Since $\Cut^\square(G)$ is very ample, it has only finitely many gaps. Let $x$ be a largest gap, i.e. a gap that belongs to the largest dilation $k$. Then $2x$ and $3x$ belong to larger dilations, so they are not gaps. Since $\Cut^\square(G)$ is seminormal, $x$ is also not a gap. A contradiction.
\end{proof}

We show how a part of Conjecture \ref{Conjecture1} is equivalent to the four color theorem.

\begin{theorem}\label{thm:equiv}
	The fact that every lattice point in $3\Cut^\square(G)$ is a sum of $3$ lattice points from $\Cut^\square(G)$ for a planar graph $G$ is equivalent to the four color theorem. 
\end{theorem}

\begin{proof}
	One implication, proving the four color theorem, is presented in \cite[Proposition 9.4]{Mi18}, but originally the idea is due to David Speyer. We note that this implication only uses a decomposition of one specific point in $3\Cut^\square(G)$.
	
	For the other implication we extend the assertion to loopless multigraphs and proceed by induction on the number of edges. Let $p\in 3\Cut^\square(G)$ be a lattice point. Let $E_0$ be the set of edges $e\in E(G)$ such that $p(e)=0$. Consider the contraction $G':=G/E_0$. Notice that $G'$ may have multiple edges, but it is loopless. Indeed, if $e\in E(G)$ became a loop in $G'$, then there was a path between endpoints $x,y$ of $e$ consisting of edges from $E_0$. This is impossible, as since $p(e)>0$ and $p$ is a convex combination of cuts, points $x,y$ were separated by some cut. Now, we may identify $p$ with a point in $3\Cut^\square(G')$ and conclude by induction.  Similarly, if there is an edge $e\in G$ such that $p(e)=3$ we may take a cut $A|B$ that separates $e$ and apply the isomorphism $\varphi_{A|B}$. We have $\varphi_{A|B}(p)(e)=0$. Hence, by induction $\varphi_{A|B}(p)$ is a sum of three cuts and so is $p$. 
	
	We are left with the case when for each edge $e$ we have $p(e)=1$ or $p(e)=2$. We claim that the set $E_1$ of edges $e$ for which $p(e)=1$ forms precisely the edges of some cut $A|B$. This is equivalent to the fact that on the dual graph $G^*$ the set $E_1$ is a cycle -- by which we mean an edge disjoint union of circuits, in the language of matroid theory, or simple cycles, in the language of graph theory. As $p$ is a lattice point, for any cut of $G^*$ the sum of values $p(e)$ over the cut is even. The same must be true for $E_1$. In particular, every vertex $v\in G^*$ must be adjacent to an even number of edges of $E_1$. By the Euler cycle argument, we know that $E_1$ is a cycle, which finishes the proof of the claim.
	
	Let $V=V_1\sqcup V_2\sqcup V_3\sqcup V_4$ be the four coloring of $G$. Since $G$ is loopless
	\[(2,2,\dots,2)=\delta_{V_1\sqcup V_2|V_3\sqcup V_4}+\delta_{V_1\sqcup V_3|V_2\sqcup V_4}+\delta_{V_1\sqcup V_4|V_2\sqcup V_3}.\]
	Therefore, after applying $\varphi_{A|B}$ we get the assertion of the theorem
	\[p=\varphi_{A|B}(\delta_{V_1\sqcup V_2|V_3\sqcup V_4})+\varphi_{A|B}(\delta_{V_1\sqcup V_3|V_2\sqcup V_4})+\varphi_{A|B}(\delta_{V_1\sqcup V_4|V_2\sqcup V_3}).\]
\end{proof}

\begin{corollary}
Let $G$ be a planar graph.
For $k=1,2,3$ every lattice point of $k\Cut^\square(G)$ is a sum of $k$ lattice points of $\Cut^\square(G)$. 
\end{corollary}
\begin{proof}
The case $k=1$ is obvious. The case $k=3$ follows from the four color theorem by Theorem \ref{thm:equiv}. For $k=2$, as in the proof of Theorem \ref{thm:equiv}, we may reduce it to the case where  on every edge the point has value one. As in the proof of Theorem \ref{thm:equiv}, this must be a cut point, and as a consequence the graph must be bipartite. We add $\delta_{\emptyset|V}=0$ to obtain a sum of precisely two lattice points of $\Cut^\square(G)$.
\end{proof}

\section*{Acknowledgements}
We would like to thank Seth Sullivant and Hidefumi Ohsugi for valuable information about cut polytopes. 


\end{document}